\documentclass[12pt,a4paper,reqno]{amsproc}

\usepackage{graphicx}
\usepackage{amsmath}
\usepackage{amssymb}
\usepackage{epstopdf}
\usepackage[backrefs]{amsrefs}

\newtheorem{thm}{Theorem}[section]
\newtheorem{cor}[thm]{Corollary}
\newtheorem{lem}[thm]{Lemma}
\newtheorem{prop}[thm]{Proposition}

\theoremstyle{remark}
\newtheorem{rem}[thm]{Remark}
\newtheorem{ex}[thm]{Example}

\theoremstyle{definition}

\newtheorem{defi}[thm]{Definition}

\title{On Projections of Free Semialgebraic Sets}

\author{Tom Drescher}
\address{T.D., Universit\"at Innsbruck, 6020 Innsbruck, Austria}
\email{tom.drescher@uibk.ac.at}

\author{Tim Netzer}
\address{T.N., Universit\"at Innsbruck, 6020 Innsbruck, Austria}
\email{tim.netzer@uibk.ac.at}
\author{Andreas Thom  }
\address{A.T., TU Dresden, 01062 Dresden, Germany}
\email{andreas.thom@tu-dresden.de}
\date{\today}                                           

\begin{document}
\maketitle

\begin{abstract}An important result in real algebraic geometry is the projection theorem: every projection of a semialgebraic set is again semialgebraic. This theorem and some of its conclusions lie at the basis of many other results, for example the decidability of the theory of real closed fields, and almost all Positivstellens\"atze. Recently, non-commutative real algebraic geometry has evolved as an exciting new area of research, with many important applications. In this paper we examine to which extend a projection theorem is possible in the non-commutative (=free) setting. Although it is not yet clear what the correct notion of a free semialgebraic set is, we review and extend some results that count against a full free projection theorem. For example, it is undecidable whether a free statement holds for all matrices of at least one size. We then prove a weak version of the projection theorem: projections along linear and separated variables yields a semi-algebraically parametrised free semi-algebraic set.
\end{abstract}

\tableofcontents

\section{Introduction and Preliminaries} The {\it projection theorem} in real algebraic geometry is a basic but utmost important result on  semialgebraic sets. A semialgebraic set is defined as a Boolean combination of sets of the form $$W(p)=\left\{ a\in\mathbb R^n\mid p(a)\geq 0\right\}$$ where $p\in\mathbb R[x_1,\ldots x_n]$ is a multivariate polynomial. The projection theorem states that any projection (and thus any polynomial image) of a semialgebraic set is again semialgebraic. This implies for example that closures, interiors, convex hulls  etc.\ of semialgebraic sets are again semialgebraic. 

Proofs for the projection theorem can be found for example in \cite{bcr, pd}. When analyzing them, it turns out that the semialgebraic description of a projection can be obtained in an explicit and uniform way from the input polynomials that define the initial set. In particular, when evaluated over any real closed extension field of the reals, the projection of the inital set is still defined by the same semialgebraic formula as over $\mathbb R$. Since projections correspond to existential quantifiers in formulas, this immediately leads to {\it quantifier elimination} over real closed fields: for any first order formula in the language of ordered rings, there is a quantifier-free formula, which is equivalent over any real closed field. From this it is finally only a small step to the {\it Tarski-Seidenberg transfer principle}: any two real closed fields fulfill the same first order formulas in the language of ordered rings. This implies decidability of the first order theory of real closed fields. It is also at the core of Artin's proof of Hilbert's 17th Problem, and indeed of almost every {\it Positivstellensatz} until today.

 A recent and flourishing area of research in real algebra and geometry concerns {\it non-commutative} semialgebraic sets. Instead of points of $\mathbb R^n$, polynomials are evaluated at Hermitian matrix tuples, and $\geqslant 0$ means that the resulting matrix is positive semidefinite. Such sets appear in many applications, for example in linear systems engineering, quantum physics, free probability and semidefinite optimization (see \cite{hkc1} for an overview). Given the profound importance of the projection theorem in the classical setup, a clarification of its status in the non-commutative context is clearly necessary. Before we explain the few existing results, let us introduce the non-commutative (= {\it free}) setup in detail. 
 
 Let $\mathbb C\langle x_1,\ldots, x_n\rangle$ be the polynomial algebra in non-commuting variables. An element  is a $\mathbb C$-linear combination of words in the letters $x_1,\ldots,x_n$, where the order of the letters does matter. We consider the involution $*$ on $\mathbb C\langle x_1,\ldots, x_n\rangle$ that is uniquely defined by $x_i^*=x_i$ for all $i$, and the fact that $*$ is complex conjugation on $\mathbb C$. By $\mathbb C\langle x_1,\ldots, x_n\rangle_h$ we denote the $\mathbb R$-subspace of {\it Hermitian elements}, i.e. fixed points of the involution. 
 
 Free polynomials  can be evaluated at tuples of square matrices; the result is a matrix of the same size. Note that the constant term is multiplied with the identity matrix of the correct size.  If $p\in \mathbb C\langle x_1,\ldots, x_n\rangle_h$ and $A_1,\ldots,A_n\in \mathbb M_s(\mathbb C)_h$ are Hermitian matrices, then $p(A_1,\ldots, A_n)\in \mathbb M_s(\mathbb C)_h$ is again Hermitian. It thus makes sense to define \begin{align*}W_s(p)&:= \left\{ (A_1,\ldots,A_n)\in \mathbb M_s(\mathbb C)_h^n\mid p(A_1,\ldots,A_n)\geqslant 0\right\} \\ O_s(p)&:= \left\{ (A_1,\ldots,A_n)\in \mathbb M_s(\mathbb C)_h^n\mid p(A_1,\ldots,A_n)> 0\right\},\end{align*}  where $\geqslant 0$ and $>0$ denote positive semidefinite- and positive definiteness. This can be defined for any matrix size $s$, and in many applications the size of the matrices is not bounded a priori. So it is convenient to consider the whole collection \begin{align*}W(p)&:= \bigcup_{s\geq 1} W_s(p) \quad \mbox{ and } \quad O(p):= \bigcup_{s\geq 1} O_s(p).\end{align*}
A slight generalization is often useful. Consider the matrix algebra $\mathbb M_d\left(\mathbb C\langle x_1,\ldots,x_n\rangle\right)$ of {\it free matrix polynomials} of size $d$. A matrix polynomial can still be evaluated at tuples of matrices; if we plug in matrices of size $s$, the result will be of size $ds$. The involution extends canonically to matrix polynomials, by transposing matrices and applying $*$ entrywise. Again we denote by $\mathbb M_d\left(\mathbb C\langle x_1,\ldots,x_n\rangle\right)_h$ the space of Hermitian elements. Hermitian matrix polynomials evaluated at Hermitian matrices result in Hermitian matrices. We can thus define $W_s(p), O_s(p), W(p)$ and $O(p)$ just as before. 
 
 We will take these sets $W(p)$ and $O(p)$ as building blocks of free semialgebraic sets.  Here we already see a significant difference between the commutative and the non-commutative setting. Whereas $O(p)$ is just the complement of $W(-p)$ in the commutative case, this is clearly not true in the free setup. So even if we allow for Boolean combinations, we have to include the sets $O(p)$ explicitly. Realizing this, there is clearly more room for possible definitions of semialgebraic sets. One could for example allow to use determinant and trace to define sets, or more subtle conditions on the eigenvalues of matrices.   It seems that a good general notion of semialgebraic sets has not been proposed in the literature so far. We do not have an answer to this problem, however, our counterexamples and results will hopefully help clarifying the problem in the future.
 
 Projections of free sets are defined in the straightforward way, by mapping a matrix tuple $(A_1,\ldots, A_n)$ to $(A_1,\ldots, A_{n-k})$, say.
  In the following section we explain why a general projection theorem as in the commutative setting can probably not be expected for free semialgebraic sets. We discuss the few existing negative results, and provide some new examples and constructions. We also prove that checking whether a free formula holds for matrices of some size is an undecidable problem. So even under very general notions of semialgebraicity, a projection theorem will probably fail. In the third section we will then prove with Theorem \ref{el} and \ref{pos2} two weak projection theorems, however under strong additional assumptions. Variables that occur only separated from the others can be eliminated; the results is described by infinitely many inequalities, parametrized in a nice semialgebraic way however. We hope that the results will eventually lead to a suitable notion of semi-algebraically parametrised free semi-algebraic sets which are closed under certain projections, see the remarks after Theorem \ref{pos}.
  
\section{Counterexamples, Undecidability, and Speculations}

To the best of our knowledge, there are two negative results on projections of free semialgebraic sets. In \cite{hm1} it is shown that there exists a {\it linear} matrix polynomial  $p$, such that a projection of $O(p)$ cannot be realized as a finite union of intersections of sets $O(q)$. This in fact contrasts the commutative case, where this is always true, due to the {\it Finiteness Theorem} (see Section 2 of \cite{pd} and the many references therein). Currently, it is not clear whether the projection  might still be semialgebraic under a suitable generalized notion of semialgebraic sets.

The second negative result is from \cite{put}. Translated to our setting it is the following: With sets $W(p)$, Boolean combinations, and projections, one can construct the following (one-dimensional) free  set $X=\bigcup_{s\geq 1} X_s\colon$ $$X_s=\left\{ -sI_s, -(s-1)I_s,\ldots, -I_s,0,I_s,\ldots, (s-1)I_s, sI_s\right\}\subseteq \mathbb M_s(\mathbb C)_h.$$ The main idea here is to codify the relations of $sl_2(\mathbb C)$ into inequalities for matrices, and use that certain elements then have only integer eigenvalues. This second example imposes quite severe obstructions to a projection theorem. For example, whether $\lambda I_s$ belongs to a Boolean combination of sets $W_s(p)$ and $O_s(p)$ is independent of the size $s$. Note that this example also opens the way to speak about integers, and could maybe be used to formulate undecidable statements. However, if an arithmetic statement $\exists x \forall y \cdots $ is formalized as $\exists x\in X\  \forall y\in X \cdots$ it will have a different meaning when evaluated at a fixed matrix size. When evaluated at matrices of size $s$, the integer $y$ will be at most $s$, and will thus not exhaust all integers. We can however use a deep result from group theory to prove an undecidability theorem for free formulas  (see for example \cite{co} for details on computability theory):

\begin{thm}\label{undec}
The question whether a free closed formula holds for at least one size of matrices is undecidable.\end{thm}
\begin{proof}
It is shown in \cite{bw} that there exists a recursively enumerable sequence of finitely presented groups $\left(G_i\right)_{i\in\mathbb N}$, such that the set $$\left\{ i\in\mathbb N\mid G_i \mbox{ has a nontrivial finite-dimensional unitary representation}\right\}$$ is not recursive (=decidable). The statement that a finitely presented group has a nontrivial unitary representation can easily be expressed as a free closed formula. So we obtain a recursively enumerable sequence of free closed formulas $\left(\varphi_i\right)_{i\in\mathbb N},$ for which there is no algorithm that decides for each $\varphi_i$, whether it holds for at least one size of matrices.
\end{proof}

\begin{rem}
If $\varphi$ is a free closed formula, the set $$\mathcal I_\varphi:=\left\{ s\in\mathbb N\mid \varphi \mbox{ holds for matrices of size } s\right\}$$ is a recursive set. This follows from commutative quantifier elimination for each fixed  matrix size. Since each recursive set is representable in Peano arithmetic, there exists a formula $\psi$ in the first order language of arithmetic, such that for any $s\in\mathbb N$, \begin{align*}s\in\mathcal I_\varphi \ &\Leftrightarrow\ {\rm PA} \vdash \psi(s) \\  s\notin\mathcal I_\varphi \ &\Leftrightarrow\ {\rm PA} \vdash \neg\psi(s)\end{align*}where ${\rm PA}$ denotes the Peano axioms for arithmetic.  The statement $\exists s\psi$ however will sometimes be independent from ${\rm PA},$ which is a direct consequence of Theorem \ref{undec}.

\end{rem}

Theorem \ref{undec} also provides severe obstructions to a projection theorem. Even if every formula is equivalent to a quantifier-free formula (in whatsoever free language), these formulas must either be undecidable in the same sense, or they cannot be found in an algorithmic way.

\vspace{0.1cm}

We will finish this section with some more explicit constructions, indicating what can be expressed with quantifiers. Note that we will not even use $\geqslant$ in the following, just polynomial expressions, $=$, Boolean combinations and projections. 

Let $X$ be a Hermitian matrix. We can express the statement, that $X$ has rank $1$ as follows:
\[
\operatorname{rk}(X)=1\quad\Leftrightarrow\quad X\neq 0\wedge \forall\,Y\left(YXY=0\vee \exists\,Z\colon ZYXYZ=X\right).
\]
The implication $\Rightarrow$ is easy to see. For $\Leftarrow$ note that $X$ has at least rank $1$ because $X\neq 0$. Therefore there is an eigenvalue $\lambda\neq 0$ of $X$ and a corresponding eigenvector $v$. Now for $Y=vv^*$ we have $YXY\neq 0$ and hence $ZYXYZ=X$ for some $Z$. That gives us
\[
\operatorname{rk}(X)=\operatorname{rk}(ZYXYZ)\leq\operatorname{rk}(YXY)\leq 1.
\]
Together we get $\operatorname{rk}(X)=1$. 

We can further express the statement, that the trace of $X$ vanishes. It is well known, that this is equivalent to $X$ being a commutator of two matrices $Y,Z$. However those two matrices are not necessarily Hermitian. We can solve this problem by writing $Y$ and $Z$ as $Y=Y_1+iY_2$ and $Z=Z_1+iZ_2$ for Hermitian matrices $Y_1,Y_2,Z_1,Z_2$. A second problem is that the resulting polynomial is not Hermitian. This problem can be solved by observing, that for an arbitrary non-commutative polynomial $p$ the statement $p=0$ is equivalent to $p^*p=0$. 

Finally we can state that $X$ is a scalar matrix if and only if it commutes with all (Hermitian) matrices:
\[
X\text{ scalar}\quad\Leftrightarrow\quad \forall\,Y\colon i(XY-YX)=0.
\]

So far the three statements $\operatorname{rk}(X)=1$, $\operatorname{tr}(X)=0$ and $X$ is scalar were very basic. We can now use those basic statements to express something more complicated:
\begin{eqnarray*}
&&X=\operatorname{tr}(Y)\cdot I \\ 
&\Leftrightarrow & (X\text{ scalar})\wedge\exists\,P \left(P^2=P\wedge (\operatorname{rk}(P)=1)\wedge(\operatorname{tr}(PXP-Y)=0)\right).
\end{eqnarray*}
For the implication $\Rightarrow$ let $v$ be any unit vector and let $P=vv^*$. Then $P^2=P$, $\operatorname{rk}(P)=1$ and
\[
\operatorname{tr}(PXP-Y)=\operatorname{tr}(Y)\cdot\operatorname{tr}(vv^*vv^*)-\operatorname{tr}(Y)=0.
\]
For the converse note that $X=\lambda I$ for some real number $\lambda$ since $X$ is scalar. Further we have
\[
0=\operatorname{tr}(PXP-Y)=\lambda \operatorname{tr}(P)-\operatorname{tr}(Y).
\]
Since $P^2=P$ the matrix $P$ has only eigenvalues $0$ and $1$ and since $\operatorname{rk}(P)=1$ the eigenvalue $1$ has multiplicity $1$. Hence we get $\operatorname{tr}(P)=1$. That gives us $\lambda=\operatorname{tr}(Y)$ and the claim follows.

As we just noted, an $s\times s$-matrix with $P^2=P$ has only eigenvalues $0$ and $1$. Hence the trace of such a matrix is always an element of $\{0,\dots,s\}$. So consider the following formula:
\[
\operatorname{intscal}(X)\quad :\Leftrightarrow\quad \exists\,Y\left(Y^2=Y\wedge X=\operatorname{tr}(Y)\cdot I\right),
\]
 which is clearly equilvalent to $$\exists\,k\in\{0,\dots,{\rm size}(X)\}\colon\quad X=k\cdot I.$$ This is more or less the same result as in the above construction from \cite{put}. We can now use this to make statements about the matrix size. For example we can give a closed formula, which is true if and only if the matrix size is prime or one:
\begin{align*}
\forall\,X,Y,Z\colon & ((X=\operatorname{tr}(I)\cdot I)\wedge\operatorname{intscal}(Y)\wedge \\
&\operatorname{intscal}(Z)\wedge (X-YZ)(X-YZ)^*=0)\\
&\Rightarrow (Y=I\vee Z=I).
\end{align*}

\section{Some Free Elimination}

After we have seen some examples that impose obstacles to a free Projection Theorem, we want to prove some positive results. We can eliminate existential quantifiers under certain strong assumptions, and obtain a description by infinitely many inequalities, that are however parametrized in a good semialgebraic way. 

Throughout this section,  we denote by  $M^t,\overline M$ and $M^*$ the transpose, conjugate and conjugate-transpose of a matrix $M\in\mathbb M_s(\mathbb C)$, respectively. We equip the space $\mathbb M_d(\mathbb C)$ with the inner product $$\langle A,B\rangle ={\rm tr}(B^*A).$$
Any matrix-polynomial  $p\in\mathbb M_d(\mathbb C\langle x_1,\ldots,x_n\rangle)$ can be written as $$p=\sum_{\omega} P_\omega \otimes\omega,$$ where the $\omega$ are words in $x_1,\ldots, x_n,$ $P_\omega \in\mathbb M_d(\mathbb C),$ $\otimes$ denotes the Kronecker product, and the sum is finite (using the Kronecker product here is useful to see how evaluation at matrices works).  For  $W\in\mathbb M_d(\mathbb C)$ we define $$p_W:= \sum_\omega W^*P_\omega W\otimes \omega \in \mathbb M_d(\mathbb C\langle x_1,\ldots,x_n\rangle)$$  and $$p_{\langle W\rangle}:= \sum_\omega \langle P_\omega,\overline W\rangle \cdot \omega \in \mathbb C\langle x_1,\ldots,x_n\rangle.$$ 
 
 \begin{defi}
Let $S$ be a subspace of $\mathbb{M}_d(\mathbb{C})_h$. We call $S$ {\it definite}, if there exists a definite element in $S$. We call $S$ {\it indefinite}, if every element in $S\backslash\{0\}$ is indefinite.
\end{defi}

The following is our main result on projections of free semialgebraic sets:
\begin{thm}\label{el}Let $p\in\mathbb M_d(\mathbb C\langle x_1,\ldots,x_n\rangle)_h$ and let $B_1,\ldots,B_m\in\mathbb M_d(\mathbb C)_h$ such that the vector space $\operatorname{span}_{\mathbb R}\{B_1,\dots,B_m\}$ is an indefinite subspace of $\mathbb{M}_d(\mathbb{C})_h$.  Then for any Hilbert space $\mathcal H$ and any choice $T_1,\ldots, T_n\in\mathbb B(\mathcal H)_h$, the following are equivalent: \begin{itemize}\item[(i)] $\exists  S_1,\ldots, S_m\in \mathbb B(\mathcal H)_h$ such that $p(T_1,\ldots, T_n)+ B_1\otimes S_1+\cdots +B_m\otimes S_m \geqslant 0.$\item[(ii)] For all $r\in\mathbb N$ with $r\leq {\rm dim}(\mathcal H)$ and for all $W_j\in \mathbb M_{d,r}(\mathbb C)$ with $$ \sum_j W_j^*B_iW_j=0 \quad \mbox{for all} \quad i=1,\ldots,m $$ we have $$\sum_{j} p_{W_j}(T_1,\ldots, T_n) \geqslant 0.$$ \end{itemize} 
\end{thm}

Before we prove the theorem, let us comment on some of the conditions. Firstly, assume that $$S:=\operatorname{span}_{\mathbb R}\{B_1,\dots,B_m\}$$ is a definite subspace. Then clearly condition (i) is fulfilled for any choice of $T_i$, in which case the whole question is not very interesting. But of course $S$ could neither be definite nor indefinite. We will deal with  this case below. Secondly, if $\mathcal H$ is finite dimensional, we can clearly restrict to the single case $r=\dim(\mathcal H)$ in (ii). If $\mathcal H$ is infinite dimensional, we have to use all $r\in\mathbb N$. Thirdly, the length of the appearing sums can be bounded in terms of $r$. This is proven in Lemma \ref{cara} in the Appendix. 

Finally, note that the description from (ii) consists of inequalities only, but infinitely many. However, these inquealities are still parametrized well in terms of the input data. We call a set of the form appearing in (ii) a \emph{semi-algebraically parametrised free semi-algebraic set}. We do not yet dare to make a precise definition of this term, but we sincerely hope that the above theorem and Theorem \ref{pos2} will lead to a class of sets which are closed under certain projections. This then might guide further research on how to set up a model theoretic framework which satisfies a suitable form of quantifier elimination in form of a projection theorem.

\begin{proof}[Proof of Theorem \ref{el}]
By conjugating all coefficients with $W_j$ and summing up, it is obvious that (i) implies (ii). To prove the converse, first  note that both conditions (i) and (ii) are equivalent to the corresponding conditions, where $p$ is replaced by $p_A$ and $B_i$ is replaced by $A^*B_iA$ for an invertible matrix $A\in\mathbb{M}_d(\mathbb{C})$. By Lemma~\ref{ortho} from the Appendix we can choose $A\in\mathbb{M}_d(\mathbb{C})_h$ such that $A^2\in S^{\perp}$ and $A$ is invertible (we set $S:=\operatorname{span}\{ B_1,\ldots, B_m\}$). Thus, we can assume that $\operatorname{tr}(B_i)=0$ for $i=1,\dots,m$. Further note that both conditions are equivalent to the conditions, where a coefficient $P_{\omega}$ of $p$ is replaced by any element in $P_{\omega}+S$. Thus, we can assume that $P_{\omega}\in S^{\perp}$ for every word $\omega$. Finally note that we can assume that $B_1,\dots,B_m$ are orthonormal.

Now consider $$\mathcal V:=\{B_1^t,\ldots, B_m^t\}^\perp\subseteq \mathbb M_d(\mathbb C).$$ Observe that ${\rm id}_d\in\mathcal V$ follows from ${\rm tr}(B_i)=0,$ and $\mathcal V$ is closed under $*$ since each $B_i$ is Hermitian. Thus  $\mathcal V$ is an operator system in the $C^*$-algebra $\mathbb M_d(\mathbb C)$ (see for example \cite{pau} for details on operator systems). Now consider the following $*$-linear map \begin{align*} \varphi\colon \mathcal V&\rightarrow \mathbb B(\mathcal H) \\ W &\mapsto p_{\langle W\rangle}(T_1,\ldots, T_n).\end{align*}   We claim that $\varphi$ is $r$-positive, meaning that $\varphi\otimes {\rm id_{\mathbb M_{r}(\mathbb C)}}$ maps positive matrices to positive operators. So let $\left( W_{ij}\right)_{i,j}\in\mathbb M_{r}(\mathcal V)$ be positive semidefinite. So there are vectors $w_{1k},\ldots, w_{rk}\in\mathbb C^d$, such that  $$W_{ij}=\sum_k w_{ik}w_{jk}^*$$ for all $i,j.$ We now compute \begin{align*}\left( \varphi(W_{ij})\right)_{ij}&=\sum_{\omega,k} \left( \langle P_\omega, \overline{w_{ik}w_{jk}^*}\rangle \right)_{i,j}\otimes \omega(T_1,\ldots, T_n) \\ &=\sum_{\omega,k} \left( {\rm tr}(\overline w_{jk} w_{ik}^tP_\omega) \right)_{i,j}\otimes\omega(T_1,\ldots, T_n)\\ &=\sum_{\omega,k} \left( w_{ik}^tP_\omega \overline w_{jk}\right)_{i,j}\otimes\omega(T_1,\ldots, T_n)\\ &= \sum_ kp_{W_k}(T_1,\ldots,T_n), \end{align*} where $W_k$ is the matrix having $\overline w_{1k},\ldots, \overline w_{rk}$ as its columns. From assumption (ii) we see that this operator is positive semidefinite, provided $\sum_k  W_k^*B_i W_k=0$ for all $i$. But this follows easily from the fact that $W_{jk}\in\mathcal V$, i.e. $\langle B_i^t, W_{jk}\rangle =0$ for all $i,j,k$.

So as an $r$-positive map from an operator space to $\mathbb B(\mathcal H)$ (for all  $r\leq \dim \mathcal H$), $\varphi$ admits a completely positive extension $\psi\colon\mathbb M_d(\mathbb C)\rightarrow\mathbb B(\mathcal H),$ by Arveson's Extension Theorem (see again \cite{pau}). For any $U\in\mathbb M_d(\mathbb C)$ we have \begin{align*}\psi(U)&=\psi\left( U -\sum_i \langle U,B_i^t\rangle B_i^t\right) + \psi\left(\sum_i \langle U,B_i^t\rangle B^t\right) \\ &=\varphi\left( U-\sum_i\langle U,B_i^t\rangle B_i^t\right) +\sum_i\langle U,B_i^t\rangle \psi(B_i^t) \\ &= p_{\langle U\rangle}(T_1,\ldots,T_n) +\sum_i\langle U,B_i^t\rangle \psi(B_i^t).\end{align*} For the second equality we have used that $U-\sum_i\langle U,B_i^t\rangle B_i^t$ lies in $\mathcal V$, since the $B_i^t$ are orthonormal, and for the third that $\langle P_\omega, B_i\rangle =0$ for all coefficients $P_\omega$ of $p$ and all $i$.

Let $E_{kj}\in\mathbb M_d(\mathbb C)$ be the matrix with $1$ in the $(k,j)$-entry, and zeroes elsewhere. Then the Choi-matrix $E=\left(E_{jk}\right)_{j,k}\in \mathbb M_d(\mathbb M_d(\mathbb C))$ is positive semidefinite, and thus \begin{align*}0\leqslant \left( \psi(E_{jk})\right)_{j,k}&=\left(  p_{\langle E_{jk}\rangle}(T_1,\ldots,T_n)+ \sum_i\langle E_{jk},B_i^t\rangle\psi(B_i^t)\right)_{j,k}\\&= p(T_1,\ldots,T_n)+ \sum_i B_i\otimes\psi(B_i^t).\end{align*} This implies (i). \end{proof}

The next result is an elimination result for strict positivity. Here we will get rid of the assumption on  $\operatorname{span}\{B_1,\dots,B_m\}$ completely, but only later.

\begin{prop}\label{pos}
\label{el>}
Let $p\in\mathbb{M}_d(\mathbb{C}\langle x_1,\dots,x_n\rangle )_h$ and $B_1,\dots,B_m\in\mathbb{M}_d(\mathbb{C})_h$ be such that  $\operatorname{span}\{B_1,\dots,B_m\}$ is an indefinite subspace of $\mathbb{M}_d(\mathbb{C})_h$. Then for any Hilbert space $\mathcal{H}$ and any choice $T_1,\dots,T_n\in\mathbb{B}(\mathcal{H})_h$, the following are equivalent:
\begin{itemize}
\item[(i)] $\exists S_1,\dots,S_m\in\mathbb{B}(\mathcal{H})_h$ such that $p(T_1,\dots,T_n)+B_1\otimes S_1+\dots+B_m\otimes S_m > 0.$
\item[(ii)] There is some $\varepsilon >0,$ such that for all $ r\leq {\rm dim}(\mathcal H)$ and all $V_j\in\mathbb{M}_{d,r}(\mathbb{C})$ with $\sum_j V_j^*B_iV_j=0$ for all $i=1,\dots,m$ and $\sum_j V_j^*V_j={\rm id}_r$ we have
\[
\sum_j p_{V_j}(T_1,\dots,T_n) \geqslant \varepsilon\cdot  {\rm id}_r\otimes{\rm id}_{\mathcal H}.
\]
\end{itemize}
\end{prop}

Let us also comment on condition (ii) here. In case that $\mathcal H$ is finite-dimensional, there are only finitely many choices of $r$ to check. For any such $r$, the set of all possible tuples of $V_j$ fulfilling the condition is compact (again using Lemma \ref{cara} from the Appendix). Thus it is easy to see that the statement ''There is some $\varepsilon >0,\ldots$'' can simply be replaced by strict positivity: $\sum_j p_{V_j}(T_1,\dots,T_n) >0$.

\begin{proof}[Proof of Proposition \ref{pos}]
(i)$\Rightarrow$(ii): By assumption (i) there are $S_1,\dots,S_m$ and an $\varepsilon>0$ such that
\[
p(T_1,\dots,T_n)+B_1\otimes S_1+\dots+B_m\otimes S_m \geqslant \varepsilon\cdot {\rm id}_d\otimes {\rm id}_{\mathcal H}.
\]
For all $r\leq {\rm dim}(\mathcal H)$ and all $V_j\in\mathbb{M}_{d,r}(\mathbb{C})$ with $\sum_j V_j^*B_iV_j =0$ for all $i=1,\dots,m$ and $\sum_j V_j^*V_j ={\rm id}_r$ we immediately get
\[
\sum_j p_{V_j}(T_1,\dots,T_n) \geqslant \varepsilon\cdot \left(\sum_j V_j^*V_j\right)\otimes {\rm id}_{\mathcal H} =\varepsilon\cdot {\rm id}_r\otimes  {\rm id}_{\mathcal H}.
\]
For (ii)$\Rightarrow$(i) define $$q:=p-\varepsilon\cdot {\rm id}_d\otimes 1\in\mathbb{M}_d(\mathbb{C}\langle x_1,\dots,x_n\rangle ).$$ We want to apply Theorem \ref{el} to $q$. So let $W_j\in\mathbb{M}_{d,r}(\mathbb{C})$ such that $\sum_j W_j^*B_iW_j =0$ for all $i=1,\dots,m$ and let $k$ be the rank of the matrix $\sum_j W_j^*W_j$. Then there is an invertible matrix $U\in\mathbb{M}_r(\mathbb{C})$ with
\[
U^*\left(\sum_j W_j^*W_j\right)U = \left(
\begin{matrix}
I_k & 0\\
0 & 0
\end{matrix}
\right)=: P.
\]
Define
\[
V:=\left(
\begin{matrix}
I_k\\
0
\end{matrix}
\right)\in\mathbb{M}_{r,k}(\mathbb{C})
\]
and $V_j:=W_jUV$ for all $j$. Then we have
\[
\sum_j V_j^*B_iV_j = V^*U^*\left(\sum_j W_j^*B_iW_j\right)UV =0
\]
for all $i=1,\dots,m$ and
\[
\sum_j V_j^*V_j = V^*U^*\left(\sum_j W_j^*W_j\right)UV = V^*PV= I_k.
\]
It follows that the  $V_j$ fulfill the assumptions from (ii), and therefore
\[
\sum_j q_{V_j}(T_1,\dots,T_n) =\sum_j p_{V_j}(T_1,\ldots,T_n) -\epsilon \cdot \underbrace{\sum_j V_j^*V_j}_{I_k}\otimes {\rm id}_{\mathcal H}\geqslant 0.
\]
Now we claim that $W_jUP=W_jU$ for all $j$. To prove this, let $x\in\mathbb{C}^r$, $x_1=Px$ and $x_2=x-x_1$. Then we have $W_jUPx=W_jUx_1=W_jUx-W_jUx_2$. Hence, it suffices to show that $W_jUx_2=0$. We compute
\begin{eqnarray*}
&&\sum_j \langle W_jUx_2, W_jUx_2\rangle \\
&=&\left\langle U^*\left(\sum_j W_j^*W_j\right)Ux_2,x_2 \right\rangle = \langle Px_2,x_2 \rangle = \langle Px-P^2x, x_2 \rangle = 0.
\end{eqnarray*}
Since every summand on the left hand side is nonnegative, we indeed get $W_jUx_2=0$ and therefore $W_jUP=W_jU$ for all $j$. With this in hand we can further compute
\begin{align*}
\sum_j q_{W_jU}(T_1,\dots,T_n) &= \sum_j q_{W_jUP}(T_1,\dots,T_n)\\
&= (V\otimes {\rm id}_{\mathcal H})\left(\sum_j q_{V_j}(T_1,\dots,T_n)\right)(V\otimes {\rm id}_{\mathcal H})^*\\
& \geqslant 0.
\end{align*}
Since we have chosen $U$ to be invertible, we also get $$\sum_j q_{W_j}(T_1,\dots,T_n) \geqslant 0.$$
We have thus checked condition (ii)  from Theorem~\ref{el}, and can conclude that there are $S_1,\dots,S_m\in\mathbb{B}(\mathcal{H})_h$ such that
\[
q(T_1,\dots,T_n)+B_1\otimes S_1+\dots+B_m\otimes S_m \geqslant 0.
\]
This is equivalent to
\[
p(T_1,\dots,T_n)+B_1\otimes S_1+\dots+B_m\otimes S_m \geqslant \varepsilon\cdot {\rm id}_d\otimes {\rm id}_{\mathcal H} > 0,
\] the desired result.
\end{proof}

Now let us start considering the case that $S$ is neither definite nor indefinite.

\begin{lem} \label{elpsd} Let $p\in\mathbb M_d(\mathbb C\langle x_1,\ldots,x_n\rangle)_h$ and let $B\in\mathbb M_d(\mathbb C)_h$ be semidefinite. Let the columns  of $W\in\mathbb M_{d,r}(\mathbb C)$ form a basis of ${\rm ker}(B)$. Then for any Hilbert space $\mathcal H$ and any choice $T_1,\ldots, T_n\in\mathbb B(\mathcal H)_h$, the following are equivalent: 
\begin{itemize}\item[(i)] $\exists S\in \mathbb B(\mathcal H)_h$ such that $p(T_1,\ldots, T_n)+ B\otimes S > 0.$\item[(ii)] $p_{W}(T_1,\ldots, T_n) > 0.$ \end{itemize} 
\end{lem}
\begin{proof}
Again (i)$\Rightarrow$(ii) is obvious, and with $S:=\lambda\cdot{\rm id}$ the direction (ii)$\Rightarrow$(i) is an immediate corollary of Lemma \ref{ext} from the Appendix. \end{proof}

The following theorem is Proposition \ref{pos}, but without any assumption on $\operatorname{span}\{B_1,\dots,B_m\}$: 

\begin{thm}\label{pos2}
Let $p\in\mathbb{M}_d(\mathbb{C}\langle x_1,\dots,x_n \rangle )_h$ and let $B_1,\dots,B_m\in\mathbb{M}_d(\mathbb{C})_h$. Then for any Hilbert space $\mathcal{H}$ and any choice $T_1,\dots,T_n\in\mathbb{B}(\mathcal{H})_h$, the following are equivalent:
\begin{itemize}
\item[(i)] $\exists S_1,\dots,S_m\in\mathbb{B}(\mathcal{H})_h$ such that $p(T_1,\dots,T_n)+B_1\otimes S_1+\dots+ B_m\otimes S_m >0.$
\item[(ii)] There is some $\varepsilon >0,$ such that for all $ r\leq \max\{{\rm dim}(\mathcal H),d\}$ and all $V_j\in\mathbb{M}_{d,r}(\mathbb{C})$ with $\sum_j V_j^*B_iV_j=0$ for all $i=1,\dots,m$ and $\sum_j V_j^*V_j={\rm id}_r$ we have
\[
\sum_j p_{V_j}(T_1,\dots,T_n) \geqslant \varepsilon\cdot  {\rm id}_r\otimes{\rm id}_{\mathcal H}.
\]
\end{itemize}
\end{thm}
\begin{proof}
The proof that (i) implies (ii) is the same as in Proposition~\ref{el>}. For the converse, set $S=\operatorname{span}(B_1,\dots,B_m)$ and let $k=\dim S$. We will prove the implication from (ii) to (i) by induction over $k$. For $k=0$ this is obvious, by choosing $V={\rm id}_d$. Now assume $k>0$. If $S$ is definite, then we can find $\lambda_1,\dots,\lambda_m\in\mathbb{R}$ such that $\lambda_1B_1+\dots+\lambda_mB_m>0$. If we set $S_i=\lambda\cdot\lambda_i \operatorname{id}_{\mathcal{H}}$ for $\lambda\in\mathbb{R}$ large enough, then we get (i). If $S$ is indefinite, then (i) follows from Proposition~\ref{el>}. Now assume, that $S$ is neither definite nor indefinite, and let $B$ be a nonzero psd element of $S$. Further let $V\in\mathbb{M}_{d,r}(\mathbb{C})$ be a matrix such that the columns of $V$ form an orthonormal basis of $\ker(B)$. Set $q:= p_{V}\in\mathbb{M}_r(\mathbb{C}\langle x_1,\dots,x_n \rangle )_h$. Now for all $t\leq \max\{{\rm dim}(\mathcal H),r\}$ and all $V_j\in\mathbb{M}_{r,t}(\mathbb{C})$ with $\sum_j V_j^*(V^*B_iV)V_j =0$ for all $i=1,\dots,m$ and $\sum_j V_j^*V_j =I_t$ we have
\[
\sum_j q_{V_j}(T_1,\dots,T_n) = \sum_j p_{VV_j}(T_1,\dots,T_n) \geqslant \varepsilon\cdot {\rm id}_d\otimes{\rm id}_{\mathcal H},
\]
where the inequality follows from assumption (ii). Since $B$ is in the kernel of the surjective linear map
\[
S\to V^*SV, M\mapsto V^*MV,
\]
we have $\dim(V^*SV)<\dim S=k$. By induction hypothesis we find $S_1,\dots,S_m\in\mathbb{B}(\mathcal{H})_h$ such that
\begin{align*}
&\quad (V^*\otimes\operatorname{id}_{\mathcal{H}})(p(T_1,\dots,T_n)+B_1\otimes S_1+\dots+B_m\otimes S_m)(V\otimes \operatorname{id}_{\mathcal{H}})\\
&= q(T_1,\dots,T_n)+V^*B_1V\otimes S_1+\dots+V^*B_mV\otimes S_m >0.
\end{align*}
By Lemma~\ref{elpsd} there is an $R\in\mathbb{B}(\mathcal{H})_h$ such that
\[
p(T_1,\dots,T_n)+B_1\otimes S_1+\dots+B_m\otimes S_m +B\otimes R >0.
\]
Since $B\in S$, we find $\lambda_1,\dots,\lambda_m\in\mathbb{R}$ such that $B=\lambda_1B_1+\dots+\lambda_mB_m$. Now we have
\begin{align*}
&\quad p(T_1,\dots,T_n)+B_1\otimes(S_1+\lambda_1 R)+\dots+B_m\otimes (S_m+\lambda_m R)\\
&= p(T_1,\dots,T_n)+B_1\otimes S_1+\dots+B_m\otimes S_m+ B\otimes R >0.
\end{align*}
This proves (i).
\end{proof}

\begin{rem}
 It is not clear whether we can get rid of the assumption on  $\operatorname{span}\{B_1,\dots,B_m\}$ in Theorem \ref{el} as well. There is one obvious way to proceed. If a tuple $(T_1,\ldots,T_n)$ fulfills (ii) in Theorem \ref{el}, one can try  to approximate it by a tuple that even fulfills (ii) from Theorem \ref{pos2}, and thus  obtain (i) from Theorem \ref{pos2} for the approximation. So the set defined by (i)  in Theorem \ref{el} is at least dense in the one defined by (ii), in a suitable sense. One example is the following statement, which holds for {\it free spectrahedrops}.
\end{rem}

\begin{cor}
Let $A_1,\ldots,A_n,B_1,\ldots,B_m\in\mathbb M_d(\mathbb C)_h$ be such that  $e_1A_1+\cdots +e_nA_n={\rm id}_d$ for some point $e\in\mathbb R^n$. Let $\mathcal H$ be a Hilbert space and assume $T_1,\ldots,T_n\in\mathbb B(\mathcal H)_h$ fulfill the following condition:
\begin{itemize}
\item[] For all $r\in\mathbb N$ with $r\leq \max\{ {\rm dim}(\mathcal H),d\}$ and all $W_j\in \mathbb M_{d,r}(\mathbb C)$ with $ \sum_j W_j^*B_iW_j=0$ for $i=1,\ldots,m$ we have $$\sum_{j} p_{W_j}(T_1,\ldots, T_n) \geqslant 0.$$ \end{itemize} 
Then for all $\varepsilon >0$ there exist $S_1,\ldots,S_m\in\mathbb B(\mathcal H)_h$ such that $$A_1\otimes T_1+\cdots +A_n\otimes T_n +B_1\otimes S_1+\cdots +B_m\otimes S_m\geqslant-\varepsilon\cdot  {\rm id}_d\otimes{\rm id}_{\mathcal H}.$$
\end{cor}
\begin{proof}
The tuple $(T_1+\varepsilon e_1{\rm id}_{\mathcal H},\ldots, T_1+\varepsilon e_1{\rm id}_{\mathcal H})$ clearly fulfills (ii) of Theorem \ref{pos2}. The statement follows from (i) in Theorem \ref{pos2} for this tuple.
\end{proof}

The elimination results above are all very special. They apply only to {\it separated and linear variables}.
We finish the section with some remarks on how to extend this to the non-linear case, where the elimination variables can even be mixed among themselves. For this let $p\in \mathbb M_d(\mathbb C\langle x_1,\ldots, x_n\rangle)_h$ and $q\in \mathbb M_d(\mathbb C\langle y_1,\ldots, y_m\rangle)_h$. As before, we want to classify for which self-adjoint operators $T_1,\ldots,T_n$ there exists $S_1, \ldots, S_m$ such that $$p(T_1,\ldots,T_n)+q(S_1,\ldots,S_m)\geqslant 0.$$ The simple idea now is to try to replace $q$ by something linear. So assume there are $B_0,\ldots, B_r\in \mathbb M_d(\mathbb C)_h, C_0,\ldots, C_r\in\mathbb M_k(\mathbb C)_h$, such that for any Hilbert space $\mathcal H$ we have \begin{align*}&\left\{ q(S_1,\ldots, S_m)\mid S_i\in\mathbb B(\mathcal H)_h\right\}\nonumber \\\label{re} = &\left\{ B_0\otimes {\rm id}_{\mathcal H} + \sum_{i=1}^r B_i\otimes R_i\mid R_i\in\mathbb B(\mathcal H)_h, C_0\otimes{\rm id}_{\mathcal H}+\sum_{i=1}^r C_i\otimes R_i\geqslant 0\right\}.\end{align*}
So we want to realize the image of $q$ as an affine linear image of a free spectrahedron. Then let $\tilde p\in\mathbb M_{d+k}(\mathbb C\langle x_1,\ldots,x_n\rangle)_h$ be defined as $$\tilde p=\left(\begin{array}{cc}P_e+B_0 & 0 \\0 & C_0\end{array}\right)\otimes e + \sum_{\omega\neq e} \left(\begin{array}{cc}P_\omega & 0 \\0 & 0\end{array}\right)\otimes\omega.$$ It is now straightforward to check that the following are equivalent, for any Hilbert space $\mathcal H$ and any choice $T_1,\ldots, T_n\in\mathbb B(\mathcal H)_h$:
\begin{itemize}
\item[(i)] $\exists S_1,\ldots, S_m\in\mathbb B(\mathcal H)_h$ such that $p(T_1,\ldots, T_n)+ q(S_1,\ldots, S_m)\geqslant 0$
\item[(ii)] $\exists R_1,\ldots, R_r\in\mathbb B(\mathcal H)_h$ such that $$\tilde{p}(T_1,\ldots, T_n)+ \sum_{i=1}^r \left(\begin{array}{cc}B_i & 0 \\0 & C_i\end{array}\right) \otimes R_i \geqslant 0.$$ \end{itemize} In condition (ii) we can now eliminate the existential quantifiers with the above results.

\begin{ex} (i) The construction applies for  example in the case $q=\sum_{i=1}^m B_i\otimes f_i(y_i)$ for some $B_i\in\mathbb M_d(\mathbb C)_h$ and $f_i\in\mathbb R[y].$ If for example $f_i(\mathbb R)=[a_i,\infty)$, then 
\begin{eqnarray*}
&&\left\{ q(S_1,\ldots, S_m)\mid S_i\in\mathbb B(\mathcal H)_h\right\}
\\&=&\left\{ \sum_{i=1}^m B_i\otimes R_i\mid R_i\in\mathbb B(\mathcal H)_h, R_i\geqslant a_i\cdot{\rm id}\right\},
\end{eqnarray*} 
using the functional calculus of bounded self-adjoint operators. The other possibilities on $f_i$ are similar.

(ii) The construction also applies in the case that $q=B\otimes \omega$ is a single term. If one variable appears in $\omega$ with odd degree, then the set $\{ q(S_1,\ldots, S_m)\mid S_i\in\mathbb B(\mathcal H)_h\}$ coincides with $\{ B\otimes R\mid R\in\mathbb B(\mathcal H)_h\}$. Otherwise, it coincides with  $\{ B\otimes R\mid R\in\mathbb B(\mathcal H)_h, R\geqslant 0\}.$

(iii) Assume $q=B\otimes\omega + B^*\otimes \omega^*,$ where the word $\omega$ provides a surjective map $$\omega\colon \mathbb B(\mathcal H)_h^m \twoheadrightarrow \mathbb B(\mathcal H)$$ for all Hilbert spaces. We then clearly have 
\begin{eqnarray*}&&\{ q(S_1,\ldots, S_m)\mid S_i\in\mathbb B(\mathcal H)_h\}\\
&=&\{ (B+B^*)\otimes R_1 + i(B-B^*)\otimes R_2\mid R_i\in\mathbb B(\mathcal H)_h\},\end{eqnarray*} and the above construction applies.
\end{ex}

\section{Appendix}

\begin{lem}
\label{closedcone}
Let $V$ be a finite dimensional vector space and let $A,B$ be two closed cones in $V$ such that $(-A)\cap B=\{0\}$. Then the cone $A+B$ is closed.
\end{lem}
\begin{proof}
First, we will prove the following more general result: If $C$ is a bounded subset of $V$, then the sets $(C-A)\cap B$ and $(C-B)\cap A$ are bounded, too. We will prove this by contradiction. Assume $(C-A)\cap B$ is not bounded. Then there is a sequence $b_n\in (C-A)\cap B$ such that $\lVert b_n\rVert\geq n$ for all $n$. Further there are sequences $a_n$ in $A$ and $c_n$ in $C$ such that $b_n=c_n-a_n$. Since $V$ is finite dimensional and since the sequence $\lVert b_n\rVert^{-1}b_n$ is obviously bounded, there is a convergent subsequence $\lVert b_{n_k}\rVert^{-1}b_{n_k}$. Because $C$ is a bounded set, the sequence $c_{n_k}$ is bounded, too. Therefore $\lVert b_{n_k}\rVert^{-1}c_{n_k}$ converges to $0$. Because $A$ and $B$ are closed, the sequence $\lVert b_{n_k}\rVert^{-1}b_{n_k}=\lVert b_{n_k}\rVert^{-1}c_{n_k}-\lVert b_{n_k}\rVert^{-1}a_{n_k}$ converges to an element in $B$ (left hand side) and to an element in $-A$ (right hand side). Since $(-A)\cap B=\{0\}$ the sequence $\lVert b_{n_k}\rVert^{-1}b_{n_k}$ must converge to $0$. On the other hand the limit of this sequence must have norm $1$. A contradiction. So we have proved, that $(C-A)\cap B$ and by symmetry $(C-B)\cap A$ are bounded, if $C$ is bounded.

Now let $x_n$ be a sequence in $A+B$, which converges in $V$. Then there are sequences $a_n$ in $A$ and $b_n$ in $B$ such that $x_n=a_n+b_n$. Define $C=\{x_1,x_2,\dots\}$. By the result above the sequence $b_n=x_n-a_n\in (C-A)\cap B$ is bounded. Thus, it has a convergent subsequence $b_{n_k}$. Now every of the sequences $x_{n_k},a_{n_k}$ and $b_{n_k}$ is convergent. Since both $A$ and $B$ are closed, the limits of $a_{n_k}$ and $b_{n_k}$ are in $A$ and $B$, respectively. Therefore, the limit of $x_{n_k}$ and, thus, the limit of $x_n$ is in $A+B$.
\end{proof}

\begin{lem}
\label{ortho}
Let $S$ be a subspace of $\mathbb{M}_d(\mathbb{C})_h$, then $S$ is definite if and only if $S^{\perp}$ is indefinite.
\end{lem}
\begin{proof}
$\Rightarrow$: If $S$ is definite, then there exists $X\in S$ such that $X>0$. Now let $Y$ be an element of $S^{\perp}$ such that $Y\geq 0$. Then we can compute
\[
\langle \sqrt{Y}\sqrt{X},\sqrt{Y}\sqrt{X} \rangle = \operatorname{tr}(\sqrt{X}Y\sqrt{X}) = \operatorname{tr}(YX) = \langle X,Y \rangle = 0.
\]
This implies $\sqrt{Y}\sqrt{X}=0$, and since $X$ is invertible, we get $\sqrt{Y}=0$, hence, $Y=0$. Thus, every Element in $S^{\perp}\backslash\{0\}$ is indefinite.

$\Rightarrow$: Let $C$ be the conic hull of $S^{\perp}$ and all psd matrices. Then $C$ is a closed convex cone by Lemma~\ref{closedcone} and $C\cap (-C)=S^{\perp}$. Since $C\cap(-C)$ is an exposed face of $C$, there is a linear functional $f$ on $\mathbb{M}_s(\mathbb{C})^h$, which is positive on $C\backslash(C\cap(-C))=C\backslash S^{\perp}$ and vanishes on $C\cap(-C)=S^{\perp}$. Note that all nonzero psd matrices are in $C\backslash(C\cap (-C))$. Hence, $\ker(f)$ is also an indefinite subspace and $S^{\perp}\subseteq\ker(f)$. Further observe that $\ker(f)^{\perp}$ is one dimensional. Thus, we can choose a generator $X$ of $\ker(f)^{\perp}\subseteq S$. Assume $X$ would be indefinite. Then there is obviously a $Y\neq 0$ such that
\[
0=\operatorname{tr}(YXY)=\operatorname{tr}(Y^2X)=\langle X,Y^2 \rangle.
\]
This implies $Y^2\in \{X\}^{\perp}=\ker(f)$. Because $\ker(f)$ is indefinite and $Y^2$ is non-zero and psd, this is a contradiction. Now let $P$ be the projection onto the kernel of $X$. Then we have $XP=0$ and, hence,
\[
0=\operatorname{tr}(XP)=\langle P,X\rangle.
\]
So $P$ is a psd element of the indefinite subspace $\{X\}^{\perp}=\ker(f)$, which implies $P=0$. Hence, $X$ must be a definite element of $S$.
\end{proof}

\begin{lem}\label{cara}
For any $d\in\mathbb N$ there exists $t\in \mathbb N$, such that for  any finite choice $W_1,\ldots,W_N\in\mathbb M_d(\mathbb C)$ there are $V_1,\ldots,V_t\in\mathbb M_d(\mathbb C)$ with $$\sum_{i=1}^N W_i^*MW_i =\sum_{i=1}^t V_i^*M V_i$$ for all $M\in\mathbb M_d(\mathbb C)$. In fact we can choose $t=2s^4$ and the $V_i$ as positive multiples of some of the $W_i$.  \end{lem}
\begin{proof} Let $k$ be  the  dimension of $\mathbb M_d(\mathbb C)$ as a complex vectorspace and let $E_1, \ldots, E_k$ be a basis.  Let $t$ be the dimension of $\mathbb M_d(\mathbb C)^k$ as a real vector space. Thus $t=2d^4$.
Given $W_1,\ldots,W_N\in\mathbb M_d(\mathbb C)$, consider the convex cone $$C:={\rm cc}\left\{ (W_i^*E_1W_i,\ldots, W_i^*E_kW_i)\mid i=1,\ldots, N\right\} \subseteq \mathbb M_d(\mathbb C)^k.$$ By Caratheodory's Theorem, since  $$E:=\left(\sum_{i=1}^NW_i^*E_1W_i,\ldots, \sum_{i=1}^N W_i^*E_kW_i\right) \in C,$$ we find $\tilde W_1,\ldots,\tilde W_t$  (in fact among the $W_i$) and $\lambda_1,\ldots,\lambda_t\geq 0$ such that   $$E=\sum_{i=1}^t \lambda_{i} (\tilde W_i^*E_1\tilde W_i,\ldots, \tilde W_i^*E_k \tilde W_i).$$ With $V_i:=\sqrt{\lambda_i}\tilde W_i$, the result follows.
\end{proof}

\begin{lem}\label{ext}
Let $A\in\mathbb B(\mathcal H)_h, B\in\mathbb B(\mathcal K,\mathcal H)$ and $C\in\mathbb B(\mathcal K)_h$ be bounded operators on Hilbert spaces. Then the following are equivalent:
\begin{itemize}
\item[(i)] $\exists \lambda\in\mathbb R$ such that $\left(\begin{array}{cc}A & B \\B^* & C+\lambda\cdot{\rm id}_{\mathcal K}\end{array}\right)> 0$
\item[(ii)] $A > 0.$
\end{itemize}
\end{lem}
\begin{proof}
(i)$\Rightarrow$(ii) is obvious. For (ii)$\Rightarrow$(i)  assume $A\geqslant \varepsilon\cdot {\rm id}_{\mathcal H},$ choose $$\lambda=\frac{2\Vert B\Vert ^2}{\varepsilon} +\frac{\varepsilon}{2} +\Vert C\Vert $$ and compute \begin{align*}&\left\langle \left(\begin{array}{cc}A & B \\B^* & C+\lambda{\rm id}_{\mathcal K}\end{array}\right) \left(\begin{array}{c}h \\k\end{array}\right),\left(\begin{array}{c}h \\k\end{array}\right)\right\rangle\\&=\langle Ah,h\rangle +2{\rm Re}\langle Bk,h\rangle + \langle (C+\lambda{\rm id})k,k\rangle \\ &\geq \varepsilon\Vert h\Vert^2  -2\Vert B\Vert\Vert h\Vert\Vert k\Vert +(\lambda -\Vert C\Vert)\Vert k\Vert^2 \\ &\geq \varepsilon/2\left(\Vert h\Vert^2+\Vert k\Vert^2\right).\end{align*} This proves $\left(\begin{array}{cc}A & B \\B^* & C+\lambda\cdot{\rm id}_{\mathcal K}\end{array}\right)\geqslant \varepsilon/2\cdot {\rm id}_{\mathcal H\oplus\mathcal K}.$
 \end{proof}

\section*{Acknowledgments}The first and second author were supported by Grant No. P 29496-N35 of the Austrian Science Fund (FWF). The third author was supported by ERC Starting Grant No.\ 277728 and ERC Consolidator Grant No.\ 681207.  The results of this article are part of the PhD project of the first named author.


\begin{bibdiv}
\begin{biblist}

\bib{bcr}{book}{
    AUTHOR = {Bochnak, Jacek},
    AUTHOR = {Coste, Michel},
    AUTHOR = {Roy, Marie-Fran\c{c}oise},
     TITLE = {Real algebraic geometry},
    SERIES = {Ergebnisse der Mathematik und ihrer Grenzgebiete (3)},
    VOLUME = {36},
   PUBLISHER = {Springer-Verlag, Berlin},
      YEAR = {1998},
     PAGES = {x+430},
}


\bib{bw}{article}{
	AUTHOR={Bridson, Martin R.},
	AUTHOR={Wilton, Henry},
	TITLE={The triviality problem for profinite completions},
	JOURNAL={Inventiones Mathematicae},
	YEAR={to appear},
}

\bib{co}{book}{
    AUTHOR = {Cooper, S. Barry},
     TITLE = {Computability theory},
 PUBLISHER = {Chapman \& Hall/CRC, Boca Raton, FL},
      YEAR = {2004},
     PAGES = {x+409},
}

\bib{hm1}{article}{
    AUTHOR = {Helton, William},
    AUTHOR = {McCullough, Scott},
     TITLE = {Every convex free basic semi-algebraic set has an {LMI}
              representation},
   JOURNAL = {Ann. of Math. (2)},
  FJOURNAL = {Annals of Mathematics. Second Series},
    VOLUME = {176},
      YEAR = {2012},
    NUMBER = {2},
     PAGES = {979--1013},
}

\bib{hkc1}{incollection}{
    AUTHOR = {Helton, William},
    AUTHOR = {Klep, Igor},
    AUTHOR = {McCullough, Scott},
     TITLE = {Free convex algebraic geometry},
 BOOKTITLE = {Semidefinite optimization and convex algebraic geometry},
    SERIES = {MOS-SIAM Ser. Optim.},
    VOLUME = {13},
     PAGES = {341--405},
 PUBLISHER = {SIAM, Philadelphia, PA},
      YEAR = {2013},
}

\bib{pau}{book}{
    AUTHOR = {Paulsen, Vern},
     TITLE = {Completely bounded maps and operator algebras},
    SERIES = {Cambridge Studies in Advanced Mathematics},
    VOLUME = {78},
 PUBLISHER = {Cambridge University Press, Cambridge},
      YEAR = {2002},
     PAGES = {xii+300},
}


\bib{pd}{book}{
    AUTHOR = {Prestel, Alexander},
    AUTHOR = {Delzell, Charles N.},
     TITLE = {Positive polynomials},
    SERIES = {Springer Monographs in Mathematics},
   PUBLISHER = {Springer-Verlag, Berlin},
      YEAR = {2001},
     PAGES = {viii+267},
}

\bib{put}{article}{
	AUTHOR = {Putinar, Mihai},
	TITLE = { Undecidability in a free $*$-algebra},
	JOURNAL = {IMA Preprint Series},
	NUMBER = {2165},
	YEAR = {2007},
	}

\end{biblist}
\end{bibdiv}
\end{document}